\pdfoutput=1

\documentclass{article}
\usepackage{xspace}
\usepackage{url}
\usepackage{mathtools}
\usepackage{amssymb}
\usepackage{amsthm}
\usepackage{empheq}
\usepackage{latexsym}
\usepackage[shortlabels]{enumitem}
\usepackage{eurosym}
\usepackage{dsfont}
\usepackage{appendix}
\usepackage{color} 
\usepackage[unicode]{hyperref}
\usepackage{frcursive}
\usepackage[utf8]{inputenc}
\usepackage[T1]{fontenc}
\usepackage{geometry}
\usepackage{multirow}
\usepackage{todonotes}
\usepackage{lmodern}
\usepackage{anyfontsize}
\usepackage{stmaryrd}
\usepackage{natbib}
\usepackage{cleveref}
\usepackage[english]{babel}
\usepackage[english=british]{csquotes}
\usepackage{mathtools}
\usepackage{accents}
\usepackage{color}
\usepackage{comment}
\usepackage{mdframed}

\usepackage{cases}
\usepackage{bm}
\usepackage[mathscr]{euscript}

\mdfsetup{middlelinecolor=blue,middlelinewidth=2pt,linewidth=0pt,backgroundcolor=blue!15,,roundcorner=10pt}

\allowdisplaybreaks

\setcitestyle{numbers,open={[},close={]}}

\definecolor{red}{rgb}{0.7,0.15,0.15}
\definecolor{green}{rgb}{0,0.5,0}
\definecolor{blue}{rgb}{0,0,0.7}
\hypersetup{colorlinks, linkcolor={red}, citecolor={green}, urlcolor={blue}}
			
\makeatletter \@addtoreset{equation}{section}

\newtheorem{theorem}{Theorem}[section]

\newtheorem{lemma}[theorem]{Lemma}

\newtheorem{definition}[theorem]{Definition}
\newtheorem{remark}[theorem]{Remark}

\newcommand{\vertiii}[1]{{\left\vert\kern-0.25ex\left\vert\kern-0.25ex\left\vert #1 \right\vert\kern-0.25ex\right\vert\kern-0.25ex\right\vert}}

\begin{document}

\title{New proofs to measurable, predictable and optional section theorems\footnote{Funded by H.F.R.I., 3rd Call for H.F.R.I. Scholarships for PHD Canditates, project i.d. 05724.}}

\author{Stefanos {Theodorakopoulos}\footnote{Department of Mathematics, NTUA, Zografou Campus, 15780 Athens, Greece, steftheodorakopoulos@mail.ntua.gr}}

\date{\today}

\maketitle

\begin{abstract}

We present new, short and elementary proofs of the famous section theorems that are used in Stochastic Calculus. Predictable section is proved directly while measurable section is a simple corollary. Then, optional (resp. accessible) section follows from an intuitive approximation argument based on the dichotomy of predictable and total inaccessible times.

\vspace{5mm}
 \vspace{5mm}
\end{abstract}

\section{Introduction}
The measurable projection and section along with the predictable and optional section theorems are the pillars of classical general theory of processes in stochastic calculus. In most textbooks their proofs are not included as they are deemed quite difficult, see for example \cite{cohen2015stochastic}[Chapter 7 p. 159], \cite{jacod2013limit}[Chapter I.2. p.19], 
or
\cite{protter2005stochastic}[Chapter I.2. p. 7], 
That is because the proofs require descriptive set theory, 
which is something a probabilist is not very used to deal with.

Generally the $\pi_{\Omega}\footnote{$\pi_{\Omega} : \Omega \times \mathbb{R}_+ \rightarrow \Omega$ with $\pi_{\Omega}(x,y) := x,\hspace{0.2cm} \pi_{\mathbb{R}_+}: \Omega \times \mathbb{R}_+ \rightarrow \mathbb{R}_+$ with $\pi_{\mathbb{R}_+}(x,y) := y.$}-$projection of sets belonging to the product $\sigma-$algebra $\mathcal{F}\otimes \mathcal{B}(\mathbb{R}_+)$ is not $\mathcal{F}-$measurable, hence a version of the capacitability theorem is almost impossible to avoid. 
The first proofs of predictable and optional section theorems in Dellacherie \cite{dellacherie2019capacites}, Dellacherie and Meyer \cite{dellacherie1978probabilities}[Theorem 84 p. 137 and Theorem 85 p. 138] built upon the measurable section theorem \cite{dellacherie1978probabilities}[Theorem 44 p. 64]. In order to prove the latter a sufficient introduction to Choquet’s capacitability theorem and analytics sets was needed, using among others auxiliary spaces, compact pavings and projections of appropriate classes. The same procedure was followed, more or less, by He, Wang, Yan \cite{he2019semimartingale}[Theorem 4.7 p.114 and Theorem 4.8 p. 115] and Medvegyev \cite{medvegyev2007stochastic}[Theorem 3.33 p. 195]. 
Bichteler \cite{bichteler2002stochastic}[Section A.5 p 432] using the same tools approached the predictable and optional section theorems directly. However his method for the optional section theorem \cite{bichteler2002stochastic}[A.5.19 p 440] is problematic, due to the lack of compactness in sets of the form $\pi_{\mathbb{R}_+}((\{\omega\}\times \mathbb{R}_+) \cap\llbracket S, T\llbracket)$. 
Bass in \cite{bass2010measurability}, to make the classical presentation simpler, used a version of capacitability which he called $t$-approximability, even so, his methods did not avoid the introduction of auxiliary Hausdorff spaces and so on. Finally let us mention that Dellacherie in \cite{SPS_1969__3__115_0} 
proved a version of the capacitability theorem applicable to monotone classes. A similar exposition can be found by Lowther \cite{Blog}. The main downside of this method is the lack of intuitiveness in the proof of this capacitability theorem. Furthermore, once again the proofs of predictable and optional section theorems rest on measurable section.

In the present work we prove the predictable section theorem directly. To do this, we show that every predictable set can be written as a union of sets from $\mathcal{E}^*_\delta$,\footnote{For every class of sets $\mathcal{A}$ we denote by $\mathcal{A}_\delta : = \left\{\bigcap_{n = 1}^{\infty}A_n : A_n \in \mathcal{A}\hspace{0.1cm} \forall n \in \mathbb{N}\right\}$.} see \Cref{2.3}, that although uncountable many, they satisfy a monotone condition which allows us to approximate the predictable set as much as we want by using one of them each time. For this analysis all we really need is the introduction of Souslin schemes and their most immediate properties, given in the almost trivial \Cref{lemma 2.2}.
Additionally, our approach allows without much trouble to avoid any mention of capacities and rely only to the immediate properties of the outer measure $\mathbb{P}^*$, see \Cref{2.5}. 
Then, measurable projection and section are an immediate corollary of predictable section, see \Cref{remark 3.3}. Our final insight is that optional section also follows directly from predictable section, as long as a suitable approximation result of optional sets from predictable is used, see \Cref{4.6}.

\section{Technical Prerequisites}

\begin{definition}
Let $E$ be an arbitrary nonempty set and $\emptyset \subset \mathcal{E} \subseteq 2^E$. A function $A_{\{n_1,...,n_k\}}: \bigcup_{k = 1}^{\infty}\mathbb{N}^k \longrightarrow \mathcal{E}\cup \{E\}$ is said to be a \emph{Souslin scheme} with values in $\mathcal{E}$.
The \emph{Souslin operation} given a Souslin scheme $A_{\{n_1,...,n_k\}}$ produces the set  $A := \bigcup_{n \in {\mathbb{N}}^{\mathbb{N}}}\bigcap_{k = 1}^{\infty}A_{n_1,...,n_k}$.
The collection of all these sets is denoted by $\mathcal{S}(\mathcal{E})$ and called the \emph{Souslin class} of $\mathcal{E}$.
Finally, a Souslin scheme $A_{\{n_1,...,n_k\}}$ will be called \emph{ vertically monotone} when for every $k \in \mathbb{N}$ and $(n_1,...,n_k)\in \mathbb{N}^k$ we have $A_{n_1,...,n_k,n_{k + 1}} \subseteq A_{n_1,...,n_k}$, and \emph{horizontally monotone} when for every $k \in \mathbb{N}$ and $(n_1,...,n_k), (m_1,...,m_k)\in \mathbb{N}^k$ such that $n_1\leq m_1, ..., n_k \leq m_k$ we have $A_{n_1,...,n_k} \subseteq A_{m_1,...,m_k}$. If it is both vertically and horizontally monotone then it will be simple called monotone.
\end{definition}

\begin{lemma}\label{lemma 2.2}
Let $E$ be an arbitrary nonempty set and $\emptyset \subset \mathcal{E} \subseteq 2^E$. The following are true.
\begin{itemize}
    \item [i.] $\mathcal{S}(\mathcal{E})$ is closed with respect to countable unions and intersections.
    
    \item [ii.] If $\mathcal{E}$ is closed with respect to finite unions and intersections and $A \in \mathcal{S}(\mathcal{E})$ produced by $A_{\{n_1,...,n_k\}}$, then we can assume without loss of generality that $A_{\{n_1,...,n_k\}}$ is monotone.
\end{itemize}
\end{lemma}

\begin{proof}
For \textit{i.}, let $\{A^m\}_{m \in \mathbb{N}} \subseteq \mathcal{S}(\mathcal{E})$ and $\vartheta : \mathbb{N} \times \mathbb{N} \longrightarrow \mathbb{N}$ an increasing per coordinate bijection.\footnote{For example we can pick as $\vartheta(k,m):= \begin{cases}
    (m - 1)^2 + m - 1 + k \hspace{0.1cm}&\text{if}\hspace{0.1cm} k \leq m\\
    (k - 1)^2 + m \hspace{0.1cm}&\text{if}\hspace{0.1cm} m < k.
\end{cases}$.}
Then we can decompose $\mathbb{N}$ into a sequence of disjoint sets as $\mathbb{N} = \bigcup_{m = 1}^{\infty}\left\{\vartheta(k,m): k \in \mathbb{N}\right\}$. Hence,
$\bigcup_{m = 1}^{\infty}A^m = \bigcup_{m = 1}^{\infty}\bigcup_{n \in {\mathbb{N}}^{\mathbb{N}}}\bigcap_{k = 1}^{\infty}A_{n_1,...,n_k}^m = 
 \bigcup_{h \in {\mathbb{N}}^{\mathbb{N}}}\bigcap_{l = 1}^{\infty} D_{h_1,...,h_l},$ where 
 $D_{h_{1},...,h_{l}} := A^{\pi_2\circ\vartheta^{-1}(h_1)}_{\pi_1\circ\vartheta^{-1}(h_1),h_2,...,h_{l}}$\footnote{$\pi_{1} : \mathbb{N} \times \mathbb{N} \rightarrow \mathbb{N}$ with $\pi_{1}(k,m) := k,\hspace{0.2cm} \pi_{2}: \mathbb{N} \times \mathbb{N} \rightarrow \mathbb{N}$ with $\pi_{2}(k,m) := m.$}
, and similarly
$
\bigcap_{m = 1}^{\infty}A^m = \bigcap_{m = 1}^{\infty}\bigcup_{n \in {\mathbb{N}}^{\mathbb{N}}}\bigcap_{k = 1}^{\infty}A_{n_1,...,n_k}^m = \bigcup_{n \in {\mathbb{N}}^{\mathbb{N}}}\bigcap_{m = 1}^{\infty}\bigcap_{k = 1}^{\infty}A_{n_{\vartheta(1,m)},...,n_{\vartheta(k,m)}}^{m} =
 \bigcup_{h \in {\mathbb{N}}^{\mathbb{N}}}\bigcap_{l = 1}^{\infty} D_{h_1,...,h_l},
$
where 
$D_{h_1,...,h_l} := A_{h_{\vartheta(1,\pi_2\circ\vartheta^{-1}(l))},...,h_{\vartheta(\pi_1\circ\vartheta^{-1}(l),\pi_2\circ\vartheta^{-1}(l))}}^{\pi_2\circ\vartheta^{-1}(l)}$.

For \textit{ii.}, let $A = \bigcup_{n \in {\mathbb{N}}^{\mathbb{N}}}\bigcap_{k = 1}^{\infty}A_{n_1,...,n_k}$, 
we define
$D_{h_1,...,h_l} := \bigcup_{\{n \in \mathbb{N}^{l} : n_1\leq h_1, ..., n_l\leq h_l\}} \bigcap_{k =1}^l A_{n_1,...,n_k}$. Obviously $D_{\{h_1,...,h_l\}}$ is monotone and $\bigcup_{n \in {\mathbb{N}}^{\mathbb{N}}}\bigcap_{k = 1}^{\infty}A_{n_1,...,n_k} \subseteq \bigcup_{h \in {\mathbb{N}}^{\mathbb{N}}}\bigcap_{l = 1}^{\infty}D_{h_1,...,h_l}$. For the reverse relation, let $h^* \in \mathbb{N}^{\mathbb{N}}$ and $x \in \bigcap_{l = 1}^{\infty}D_{h^*_1,...,h^*_l}$, then for every $l \in \mathbb{N}$ there exists a $(n_1^l,...,n^l_l) \in \mathbb{N}^l$ such that $n^l_1 \leq h_1^* ,..., n^l_l \leq h_l^*$ and $x \in \bigcap_{k=1}^{l}A_{n^l_1,...,n^l_k}$. Due to the finiteness of $\{1,..., h_1^*\}$ there exists an $n^*_1 \in \{1,..., h_1^*\}$ such that for infinite $l \in \mathbb{N}$ we have $n^l_1 = n^*_1$. In the next step, again by the finiteness of $\{1,..., h_2^*\}$, there exists an $n_2^* \in \{1,...,h_2^*\}$ such that for infinite $l \in \mathbb{N}$ we have $n^l_1 = n^*_1$ and $n^l_2 = n^*_2$. Continuing this way we find a sequence $n^* \in \mathbb{N}^{\mathbb{N}}$ such that for every $k \in \mathbb{N}$ there are infinite $l \in \mathbb{N}$ where $n^l_1 = n^*_1, \dots, n^l_k = n^*_k$, hence we have $x \in \bigcap_{k = 1}^{\infty}A_{n_1^*,...,n_k^*}$.
\end{proof}

 Because the images of projections generally are not measurable we need to make a trivial extension of $\mathbb{P}$ to the whole power set. 
 
\begin{lemma}\label{2.5}
Let $(\Omega,\mathcal{F},\mathbb{P})$ be a probability space. Then, the set function $\mathbb{P}^*: 2^\Omega \rightarrow \mathbb{R}_+$ where $\mathbb{P}^*(A):= \inf\{\mathbb{P}(E): A \subseteq E$ and $E \in \mathcal{F}\}$, is an outer measure extension of $\mathbb{P}$ such that for all $A \in  2^\Omega$ there exists $E_A \in \mathcal{F}$ with $A \subseteq E_A$ and $\mathbb{P}^*(A) = \mathbb{P}(E_A)$. Finally, $\mathbb{P}^*$ is continuous on increasing sequences. 
\end{lemma}

\begin{proof}
The fact that is an extension is obvious, as is also obvious that for all $A \in  2^\Omega$ exists $E_A \in \mathcal{F}$ with $A \subseteq E_A$ and $\mathbb{P}^*(A) = \mathbb{P}(E_A)$. From this follows immediately the countable subadditivity property. Now, let $\{A_n\}_{n \in \mathbb{N}} \subseteq \mathcal{P}(\Omega)$ be an increasing sequence, then there exists a corresponding sequence $\{E_{A_n}\}_{n \in \mathbb{N}} \subseteq \mathcal{F}$ such that for all $n \geq 1$ it is true that $A_n \subseteq E_{A_n}$ and $\mathbb{P}(E_{A_n}) = \mathbb{P}^*(A_n)$. We define the sequence $\{B_n\}_{n \in \mathbb{N}}$ where $B_n:= \bigcap_{m = n}^{\infty}E_{A_n}$, this is increasing with $\mathbb{P}(B_n) = \mathbb{P}^*(A_n)$ and $\bigcup_{n = 1}^{\infty}A_n \subseteq \bigcup_{n = 1}^{\infty}B_n$. Hence, we have
$
 \mathbb{P}(\bigcup_{n = 1}^{\infty}B_n) = \lim_{n \rightarrow \infty}\mathbb{P}(B_n) = \lim_{n \rightarrow \infty}\mathbb{P}^*(A_n) \leq \mathbb{P}^*(\bigcup_{n = 1}^{\infty}A_n) \leq \mathbb{P}(\bigcup_{n = 1}^{\infty}B_n).$
\end{proof}

\section{
Section Theorems}\label{4}

Let $(\Omega, \mathcal{F})$ be a measurable space and $\mathbb{P}$ a probability measure on $\mathcal{F}$. We remind that given a filtration $\{\mathcal{F}_t\}_{t \in \mathbb{R}_+}$ (without the usual conditions) on $\Omega \times \mathbb{R}_+$ such that $\bigvee_{t \in \mathbb{R}_+}\mathcal{F}_t \subseteq \mathcal{F}$, and assuming that one is familiar with the \emph{optional} (stopping) times denoted by $\mathscr{O}$ and their basic properties, we say that an optional time $\rho$ is \emph{predictable} if and only if there exists a non decreasing sequence of optional times $\{\rho_n\}_{n \in \mathbb{N}}$ such that $\rho_n < \infty, \rho_n \leq \rho$ and $\rho_n < \rho$ on $\{0 < \rho\}$, for all $n\in \mathbb{N}$, with the property $\rho_n \nearrow \rho$. By abusing notation, as we did with the optional times and the optional $\sigma-$algebra, we denote the set of predictable times as $\mathscr{P}$. It will be clear from the context when $\mathscr{P}$ symbolizes the predictable times and when the predictable $\sigma-$algebra of $\Omega \times \mathbb{R}_+$.

\begin{definition}
For every $S \in \Omega \times \mathbb{R}_+$ the \emph{debut} of $S$ is a function $\Omega \rightarrow \mathbb{R}_+ \cup \{\infty\}$ which is denoted by $\mathscr{D}[S]$ and is defined as 
\[
\mathscr{D}[S](\omega) := \inf\{s \in \mathbb{R} : (\omega,s) \in S\},
\]
with the convention $\inf \emptyset = \infty$. In addition, for every function $\tau : \Omega \rightarrow \mathbb{R}_+ \cup \{\infty\}$ we denote by $\llbracket\tau\rrbracket$ its \emph{graph} $\{(\omega,\tau(\omega)) : \omega \in \Omega \hspace{0.1cm} \text{and} \hspace{0.1cm} \tau(\omega) < \infty\}$, and for every $A \in \mathcal{F}$ we denote by $\tau_A$ the function $\Omega \rightarrow \mathbb{R}_+\cup\{\infty\}$ such that $\tau_A := \tau \mathds{1}_A + \infty \mathds{1}_{A^c}$.
\end{definition}
Observe that above the predictable times were defined without any reference to a specific probability measure by demanding the convergence to hold for all $\omega$.
The following properties of predictable times are basic and their proofs are trivial so they are omitted.
\begin{itemize}
\item For all $\tau \in \mathscr{O}$ and $t \in \mathbb{R}_+ \setminus \{0\}$ we have that $\tau + t \in \mathscr{P}$.
\item For all $A \in \mathcal{F}_0$ we have $0_A \in \mathscr{P}$.
\item For all $\rho_1, \rho_2 \in \mathscr{P}$ we have that $\rho_1 \wedge \rho_2 \in \mathscr{P}$.
\item For every sequence $\{\rho_n\}_{n \in \mathbb{N}} \subseteq \mathscr{P}$ we have that $\sup_{n \in \mathbb{N}}\{\rho_n\} \in \mathscr{P}$.
\item For all $\tau \in \mathscr{O}$ and $\rho \in \mathscr{P}$ we have that $\rho_{\{\rho \leq \tau\}} \in \mathscr{P}$.
\end{itemize}

At this point we note that when we write $\llbracket\rho,\tau\rrbracket$ for $\rho, \tau \in \mathscr{O}$ we mean $\llbracket\rho,\infty\llbracket \hspace{0.1cm} \cap \hspace{0.1cm} \llbracket0,\tau\rrbracket$ but we do not demand $\rho \leq \tau$. Also $\llbracket\infty,\tau\rrbracket = \emptyset$, for every $\tau \in \mathscr{O}$.

\begin{lemma}\label{2.3}
Let $\mathcal{E}^* := \{\bigcup_{k = 1}^{n}\llbracket\rho_k,\tau_k\rrbracket:$ $n\in \mathbb{N},\rho_k \in \mathscr{P}$ and $\tau_k \in \mathscr{O} \hspace{0.2cm} \text{with} \hspace{0.2cm} \tau_k < \infty\} \cup \{\emptyset\}$. Then for every predictable set $P \in \mathscr{P}$ there exists a monotone Souslin scheme $P_{\{n_1,...,n_k\}}$ with values in $\mathcal{E}^*$ which produces it. 
\end{lemma}
\begin{proof}
We are going to show that $\sigma(\mathcal{E}^*)= \{D \in \mathcal{S}(\mathcal{E}^*): D^c \in \mathcal{S}(\mathcal{E}^*)\}$ and thus $\sigma(\mathcal{E}^*) \subseteq \mathcal{S}(\mathcal{E}^*)$, then because $\mathcal{E}^*$ is closed with respect to finite unions and intersections from \textit{ii.} of \Cref{lemma 2.2} we will be done. Observe that $\mathscr{P} := \{\llbracket\rho,\infty\llbracket : \rho \in \mathscr{P}\} = \sigma(\mathcal{E}^*)$. From \textit{i.} of \Cref{lemma 2.2} it is enough to show that $\mathcal{E}^* \subseteq \{D \in \mathcal{S}(\mathcal{E}^*): D^c \in \mathcal{S}(\mathcal{E}^*)\}$. So, let $\rho \in \mathscr{P}$ and $\tau \in \mathscr{O}$, then $(\llbracket \rho,\tau\rrbracket)^c = \rrbracket\tau,\infty\llbracket \cup \llbracket0,\rho\llbracket$, but $\rrbracket\tau,\infty\llbracket = \bigcup_{n = 1}^{\infty} \llbracket\tau + \frac{1}{n}, n\rrbracket$ and $\llbracket0,\rho\llbracket = \bigcup_{n = 1}^{\infty} \llbracket\frac{1}{n}, \rho_n\rrbracket \cup \llbracket0_{\{\rho > 0\}},0_{\{\rho > 0\}}\wedge 1\rrbracket$, where $\{\rho_n\}_{n \in \mathbb{N}} \subseteq \mathscr{O}$ such that $\rho_n \nearrow \rho$ as in the definition of a predictable time, hence again by \textit{i.} of \Cref{lemma 2.2} the proof is complete.
\end{proof}

\begin{theorem}[Predictable Section]\label{pr sec the}
For every predictable set $P$ in $\mathscr{P}$ and $\epsilon > 0$ there exists a predictable time $\rho^{P,\epsilon} $ such that $\llbracket\rho^{P,\epsilon}\rrbracket \subseteq P$ and $\mathbb{P}^*\left(\pi_{\Omega}(P)\right) - \mathbb{P}\left(\{\rho^{P,\epsilon} < \infty\}\right) \leq \epsilon$.
\end{theorem}

\begin{proof}
Let $P_{\{n_1,...,n_k\}}$ be a monotone Souslin scheme with values in $\mathcal{E}^*$
such that 
$P =  \bigcup_{n \in {\mathbb{N}}^{\mathbb{N}}}\bigcap_{k = 1}^{\infty}P_{n_1,...,n_k}$, this is guaranteed by \Cref{2.3}. 
So, $P$ can be written as an uncountable union of sets from $\mathcal{E}^*_\delta$ indexed by $\mathbb{N}^{\mathbb{N}}$.

Now, for every $D \in \mathcal{E}^*_\delta$, we will show that $\mathscr{D}[D] \in \mathscr{P}$ and $\llbracket \mathscr{D}[D]\rrbracket \subseteq  D$, which additionally means that $\pi_\Omega(D) = \{\mathscr{D}[D]<\infty\} \in \mathcal{F}$. First for $D \in \mathcal{E}^*$, we have $\mathscr{D}[D] = \min_{1 \leq k \leq m_D}\{{\rho^D_k}_{\{\rho^D_k \leq \tau^D_k\}}\} \in \mathscr{P}$ where $D = \bigcup_{k = 1}^{m_D}\llbracket\rho^D_k,\tau^D_k\rrbracket$ as in the definition, and because for every $\omega \in \Omega$ the $\pi_{\mathbb{R}_+}\left(\left(\{\omega\} \times \mathbb{R}_+\right) \cap D\right)$ is compact, we also get that $\llbracket \mathscr{D}[D]\rrbracket \subseteq  D$. Next, more generally, for $D = \bigcap_{k = 1}^{\infty}D_k$ where $\{D_k\}_{k \in \mathbb{N}} \subseteq \mathcal{E}^*$, note that because $\mathcal{E}^*$ is closed with respect to intersections we can assume that $\{D_k\}_{k \in \mathbb{N}}$ is decreasing. Then, for every $\omega \in \Omega$ the sequence $\{\pi_{\mathbb{R}_+}\left(\left(\{\omega\} \times \mathbb{R}_+\right) \cap D_k\right)\}_{k \in \mathbb{N}}$ is decreasing and its elements are compact, hence $ \{\mathscr{D}[D_k]\}_{k \in \mathbb{N}}$ is increasing and $\pi_\Omega(\bigcap_{k = 1}^{\infty}D_k)) = \bigcap_{k = 1}^{\infty}\pi_\Omega(D_k)$. So,
$\llbracket \sup_{k \in \mathbb{N}}\{\mathscr{D}[D_k]\}\rrbracket \subseteq \bigcap_{k=1}^{\infty}D_k$, which also implies that $\mathscr{D}[\bigcap_{k = 1}^{\infty}D_k] 
= \sup_{k \in \mathbb{N}}\{\mathscr{D}[D_k]\}\in \mathscr{P}$.

As a last step we will find an $m^* \in \mathbb{N}^\mathbb{N}$ such that 
$\mathbb{P}(\pi_\Omega(\bigcap_{k=1}^{\infty}P_{m^*_1,...,m^*_k})) \geq \mathbb{P}^*(\pi_\Omega(P)) - \epsilon$.
For every $k \in \mathbb{N}$ and $(m_1,...,m_k) \in \mathbb{N}^k$ we define
$
M_{m_1,...,m_k}:= \bigcup_{\{n \in \mathbb{N}^{\mathbb{N}} : \hspace{0.1cm}n_1 \leq m_1,...,n_k \leq m_k\}}\bigcap_{i = 1}^{\infty}P_{n_1,...,n_i}.
$
It is straightforward that if $ m_1 \rightarrow \infty$ then $M_{m_1} \uparrow P$, and more generally, if $m_{k + 1} \rightarrow \infty$ then $M_{m_1,...,m_k,m_{k + 1}} \uparrow M_{m_1,...,m_k}$. Thus $\pi_\Omega(M_{m_1,...,m_k,m_{k+1}}) \uparrow \pi_\Omega(M_{m_1,...,m_k})$ and from \Cref{2.5} we can choose a sequence $m^* \in \mathbb{N}^{\mathbb{N}}$ such that $\mathbb{P}^*(\pi_\Omega(M_{m_1^*,...,m_k^*})) > \mathbb{P}^*(\pi_\Omega(P)) - \epsilon$. 
From the monotonicity of $P_{\{n_1,...,n_k\}}$ we have $M_{m_1^*,...,m_k^*}\subseteq P_{m_1^*,...,m_k^*}$ and $\bigcap_{k=1}^{\infty}M_{m_1^*,...,m_k^*} = \bigcap_{k=1}^{\infty}P_{m_1^*,...,m_k^*}$. It follows that 
$\mathbb{P}(\pi_\Omega(\bigcap_{k = 1}^{\infty}P_{m_1^*,...,m_k^*})) = \lim_{k \rightarrow \infty}\mathbb{P}(\pi_\Omega(P_{m_1^*,...,m_k^*})) \geq \mathbb{P}^*(\pi_\Omega(P)) - \epsilon$. Hence we can define as $\rho^{P,\epsilon} := \mathscr{D}[\bigcap_{k = 1}^{\infty}P_{m_1^*,...,m_k^*}]$.
\end{proof}

\begin{remark}\label{remark 3.3}
By choosing to work with the constant filtration $\{\mathcal{F}_t\}_{t \in \mathbb{R}_+}$, where $\mathcal{F}_t = \mathcal{F}$ for all $t \in \mathbb{R}_+$, we get that $\mathscr{P} = \mathcal{F}\otimes \mathcal{B}(\mathbb{R}_+)$. To see this note that for $A \in \mathcal{F}$ and $s_1,s_2 \in \mathbb{R}_+$ the functions $\tau_1: = s_1 \mathds{1}_A + \infty \mathds{1}_{A^c}$ and $ \tau_2: = s_2 \mathds{1}_A + \infty \mathds{1}_{A^c}$ are predictable times. So, $ A \times [s_1,s_2) = \llbracket \tau_1,\infty\llbracket \hspace{0.1cm}\setminus\hspace{0.1cm} \llbracket \tau_2,\infty\llbracket\hspace{0.1cm} \in \mathscr{P}$. Then, the measurable section (and projection) theorem follows directly from \Cref{pr sec the}.
\end{remark}

\begin{definition}
An optional time $\tau$ is called
\begin{enumerate}
    \item[i.] \emph{total inaccessible} if and only if for every predictable time $\rho$ we have $\mathbb{P}\left(\tau = \rho < \infty\right) = 0$,

    \item[ii.] \emph{accessible} if and only if there exists a sequence of predictable times $\{\rho_m\}_{m \in \mathbb{N}}$ such that $\llbracket \tau \rrbracket \subseteq \bigcup_{m = 1}^{\infty}\llbracket \rho_m \rrbracket$.
\end{enumerate}
\end{definition}

Before we continue we remind the following, their proofs are trivial and so are omitted.
\begin{itemize}
    \item For every $\tau \in \mathscr{O}$ and progressively measurable set $S$ the function with graph $\llbracket\tau\rrbracket \cap S$ is an optional time.
    
    \item Every set $O \in \mathscr{O}$ is progressively measurable.

    \item For every $\tau \in \mathscr{O}$ there exist $\tau^1$ total inaccessible and $\tau^2$ accessible such that $\llbracket \tau \rrbracket = \llbracket \tau^1 \rrbracket \cup \llbracket \tau^2 \rrbracket$. 
\end{itemize}

\begin{remark}
From the last bullet it is straight forward that, for every $\tau \in \mathscr{O}$ there exists a total inaccessible time $\tau^1$ and a sequence of predictable times $\{\rho_m\}_{m \in \mathbb{N}}$ such that $\llbracket \tau \rrbracket \subseteq \llbracket \tau^1\rrbracket \cup \left(\bigcup_{m = 1}^{\infty}\llbracket\rho_m\rrbracket\right).$
\end{remark}

\begin{definition}
A set $S \subseteq \Omega \times \mathbb{R}_+$ is called \emph{thin set} if and only if there exists a sequence of optional times $\{\tau_n\}_{n \in \mathbb{N}}$ such that $S = \bigcup_{n = 1}^{\infty}\llbracket\tau_n\rrbracket$. Specifically, $S$ is called \emph{total inaccessible thin set} if and only if every $\tau_n$ can be chosen to be total inaccessible.
\end{definition} 

\begin{lemma}\label{4.6}
For every optional set $O \in \mathscr{O}$ there exists a predictable set $P \in \mathscr{P}$ such that $O \setminus P$ is a thin set and $P\setminus O$ is a total inaccessible thin set.    
\end{lemma}

\begin{proof}
For the optional $\sigma-$algebra we have that $\mathscr{O} := \sigma\left(\left\{\llbracket\tau,\infty\llbracket\hspace{0.1cm} : \tau \hspace{0.1cm}\text{optional time}\right\}\right)$. Obviously, for every $\tau$ optional time we have $\llbracket\tau,\infty\llbracket \hspace{0.1cm} \setminus \hspace{0.1cm} \rrbracket\tau,\infty\llbracket = \llbracket\tau\rrbracket$. So, because the family $\left\{\llbracket\tau,\infty\llbracket \hspace{0.1cm}: \tau \hspace{0.1cm}\text{optional time}\right\}$ is closed with respect to finite intersections, by an easy Dynkin class argument we have that exist a predictable set $P'$ and a thin set $S'$ such that $O \triangle P'\footnote{$O \triangle P := \left(O \setminus P\right) \cup \left(P \setminus O\right).$} \subseteq S'$. Because $O \triangle P'$ is progressively measurable it follows that is also a thin set. Next, let $\{\tau^1_n\}_{n \in \mathbb{N}}$ and $\{\rho_m\}_{m \in \mathbb{N}}$ be sequences of total inaccessible and predictable times respectively such that $O \triangle P' \subseteq \left(\bigcup_{n = 1}^{\infty}\llbracket \tau^1_n \rrbracket\right) \cup \left(\bigcup_{m = 1}^{\infty}\llbracket\rho_m\rrbracket\right)$. It is immediate that the set $P := P' \setminus \left(\bigcup_{m = 1}^{\infty}\llbracket\rho_m\rrbracket\right)$ satisfies what we want.
\end{proof}

\begin{theorem}[Optional Section]\label{4.7}
For every optional set $O \in \mathscr{O}$ and $\epsilon > 0$ there exists an optional time $\tau^{O,\epsilon}$ such that $\llbracket \tau^{O,\epsilon}\rrbracket \subseteq O$ and $\mathbb{P}^*\left(\pi_{\Omega}(O)\right) - \mathbb{P}\left(\{\tau^{O,\epsilon} < \infty\}\right) \leq \epsilon$.
\end{theorem}

\begin{proof}
From \Cref{4.6} there exists a predictable set $P$ such that $O \setminus P$ is a thin set and $P\setminus O$ is a total inaccessible thin set.
From \Cref{pr sec the} there is a predictable stopping time $\rho^{P,\frac{\epsilon}{2}}$ such that $\llbracket \rho^{P,\frac{\epsilon}{2}}\rrbracket \subseteq P$ and $\mathbb{P}^*\left(\pi_{\Omega}(P)\right) - \mathbb{P}\left(\{\rho^{P,\frac{\epsilon}{2}} < \infty\}\right) \leq \frac{\epsilon}{2}$. Let $\{\tau_n\}_{n \in \mathbb{N}}$ be optional times such that $O \setminus P = \bigcup_{n = 1}^{\infty} \llbracket \tau_n \rrbracket$. We have $\pi_{\Omega}(O \setminus P) = \bigcup_{n = 1}^{\infty}\left\{\tau_n < \infty\right\} \in \mathcal{F}$. So, we pick $N \in \mathbb{N}$ large enough such that $\mathbb{P}\left(\pi_{\Omega}(O \setminus P)\right) - \mathbb{P}(\bigcup_{n = 1}^{N}\left\{\tau_n < \infty\right\}) \leq \frac{\epsilon}{2}$. Next, we define as $\rho$ the optional time with graph $\llbracket \rho^{P,\frac{\epsilon}{2}} \rrbracket \cap \left(P\setminus O\right)^c$ and $\tau := \min_{n \in \{1,...N\}}\{\tau_n\}$. We have that $\mathbb{P}^*\left(\pi_{\Omega}(O \cap P)\right) - \mathbb{P}\left(\{\rho < \infty\}\right) \leq \frac{\epsilon}{2}$, and \begin{align*}
\mathbb{P}^*\left(\pi_{\Omega}(O)\right) &\leq \mathbb{P}^*\left(\pi_{\Omega}(O \cap P)\right) + \mathbb{P}\left(\pi_{\Omega}(O \setminus P) \setminus \left(\{\rho < \infty\} \cap \{\tau < \infty\}\right)\right) \\
&\leq \mathbb{P}\left(\{\rho < \infty\}\right) + \mathbb{P}\left(\{\tau < \infty\} \setminus \{\rho < \infty\}\right) + \epsilon = \mathbb{P}\left(\{\rho < \infty\}\cup\{\tau < \infty\}\right) + \epsilon.
\end{align*}
Hence, we can define as $\tau^{O,\epsilon} := \rho \wedge \tau$.
\end{proof}

\begin{remark}
One can prove the accessible section theorem, which is the analog of \Cref{4.7} for the accessible $\sigma-$algebra $\mathcal{A} := \sigma\left(\left\{\llbracket\tau,\infty\llbracket: \tau \hspace{0.1cm}\text{accessible time}\right\}\right)$ and the accessible times, exactly the same way as above.
\end{remark}

\end{document}